\documentclass[12pt]{article}

\usepackage{amsmath,amssymb,amsthm,enumitem,graphicx,hyperref}

\hypersetup{
	hidelinks,
	colorlinks,
	urlcolor=blue,
	linkcolor=[rgb]{.4,0,0},
	citecolor=[rgb]{0,.3,0},
	pdfstartview=FitH,
	pdfpagemode=UseNone
}

\newcommand{\pr}{\mathbb{P}}
\newcommand{\expec}{\mathbb{E}}
\newcommand{\ind}{\mathbf{1}}

\title{\normalsize\bfseries
	ON THE SLOW PHASE FOR FIXED-ENERGY ACTIVATED RANDOM WALKS
}

\author{
	Bernardo N. B. de Lima\footnote{Departamento de Matemática,
		Universidade Federal de Minas Gerais, Belo Horizonte, Brazil.
		Email: bernardonblima@gmail.com}
	\and
	Leonardo T. Rolla\footnote{Instituto de Matemática e Estatística,
		Universidade de São Paulo, São Paulo, Brazil.
		Email: leonardo.rolla@gmail.com}
	\and
	Célio Terra\footnote{Departamento de Matemática, Universidade Federal de Minas Gerais, Brazil.
		Current address: Instituto de Matemática, Universidade Federal do Rio de Janeiro, Brazil.
		Email: caugusto.terra@gmail.com}
}

\date{}

\begin{document}
	
	\maketitle
	
	\vspace{-1em}
	
	\begin{center}
		\begin{minipage}{13cm}
			\small
			\textbf{Abstract.}
			We study the Activated Random Walk model on the one-dimensional ring, in the high density regime. We develop a toppling procedure that gradually builds an environment that can be used to show that activity will be sustained for a long time. This yields a self-contained and relatively short proof of existence of a slow phase for arbitrarily large sleep rates.
			
			
		\end{minipage}
	\end{center}
	
	\vspace{1em}
	
	% ---------- Section style ----------
	\renewcommand{\thesection}{\arabic{section}.}
	\renewcommand{\thesubsection}{\arabic{section}.\arabic{subsection}}
	
	\makeatletter
	\renewcommand\section{\@startsection
		{section}{1}{\z@}
		{-2.5ex \@plus -0.5ex \@minus -.2ex}
		{1.2ex \@plus .2ex}
		{\normalfont\normalsize\bfseries}}
	\renewcommand\subsection{\@startsection
		{subsection}{2}{\z@}
		{-2ex \@plus -0.5ex \@minus -.2ex}
		{1ex \@plus .2ex}
		{\normalfont\normalsize\bfseries}}
	\makeatother
	
	% ---------- Theorem environments ----------
	\newtheorem{theorem}{Theorem}
	\newtheorem{proposition}{Proposition}
	\newtheorem{lemma}{Lemma}
	\newtheorem{condition}{Condition}
	\newtheorem{definition}{Definition}
	\newtheorem{remark}[theorem]{Remark}
	% ---------- Start of paper ----------
	
	\maketitle
	
\noindent\small\textit{This preprint has the same numbering of sections,
	equations and theorems as the published version in
	\emph{J. Appl. Probab.}
	\href{https://doi.org/10.1017/jpr.2026.10083}
	{doi:10.1017/jpr.2026.10083}}
	\vspace{1em}

	\section{Introduction}
	
Activated Random Walk (ARW) is an Abelian model of interacting particles that can be described as follows. On the one-dimensional ring $\mathbb{Z}_N:=\mathbb{Z}/N\mathbb{Z}$, we take a deterministic number $\zeta N$ of particles in arbitrary initial positions. There are two types of particles, \emph{active} and \emph{sleeping}. Active particles perform independent continuous-time random walks at rate $1$, and fall asleep at rate $\lambda$. Sleeping particles do not move, and continue to sleep until an active particle reaches the same site, and then become active again. See~\cite{Rolla20} for an introduction and main results, and~\cite{AsselahForienGaudilliere22, AsselahRollaSchapira22, BristielSalez22, CabezasRolla21, Forien24, ForienGaudilliere22, HoffmanHuRicheyRizzolo22, HoffmanJohnsonJunge24, Hu22, LevineSilvestri21, PodderRolla21, Taggi23} for more recent developments.

As the ring $\mathbb{Z}_N$ is finite, it is immediate that there is fixation almost surely (a.s.) if, and only if, $\zeta \le 1$. Thus, we are interested in \emph{how much time} the system will take to stabilize. Our main result is the following theorem.

\begin{theorem} \label{slow_phase} Let
	$\mathcal{J}$ be the total number of jumps the particles do until the initial configuration is stabilized.
	For every $0< \lambda <+\infty$, there are constants $\delta, \delta '>0$, depending on $\lambda$, such that, if $\zeta$ is close enough to $1$,
	$$\pr(\mathcal{J}\ge e^{\delta N})\ge 1-e^{-\delta'N},$$
	for every $N$ sufficiently large.
\end{theorem}

While the authors worked on this problem, two papers came out. In~\cite{AsselahForienGaudilliere22}, Theorem~\ref{slow_phase} is proved for the $d$-dimensional torus, $d \ge 2$, using auxiliary results and arguments from~\cite{ForienGaudilliere22}. Their proof uses a hierarchical argument: fixing a ``dormitory'' set where the particles will eventually stabilize, the model is reduced to a model of density $1$ in the dormitory, which is then divided in a series of growing ``clusters''. Stabilization of each of those clusters will affect the configuration in the other clusters, leading to a long stabilization time with high probability, even after summing over all the possible dormitory sets. A modification of this argument shows that the result is still true in the one-dimensional case.
In~\cite{HoffmanJohnsonJunge24} it is shown that the critical density for the fixed-energy model on $\mathbb{Z}_N$ and $\mathbb{Z}$ are the same. Their proof is based on the cyclical argument of~\cite{BasuGangulyHoffmanRichey19}, using a ``percolation in layers'' approach to show that, in every cycle, there will remain a fraction of active particles that will be sufficient to wake up the sleeping particles. 

Our proof uses a different approach which is self-contained and shorter.
The main strategy is to build a ``carpet'' where active particles can perform long excursions.
Unlike~\cite{HoffmanRicheyRolla23}, we do not assume a convenient initial configuration; instead, our carpet is built on the fly while the system dynamics develop.
We then divide the ring into large blocks, and prove that with high probability the system reaches certain configurations having many active particles in a certain region of the ring. The core of the argument is that, if we take the density large enough (roughly $1-C_1\exp(C_2 \lambda)$, with $C_1$ and $C_2$ constants), we can overcome the large sleeping rate by using the fact that, when a particle does not sleep, it will make a long excursion and wake up many particles.
In contrast to~\cite{BasuGangulyHoffmanRichey19}, which handles the low-$\lambda$ regime, at the end of each cycle we are left with inconvenient configurations connecting the source and sink regions, and need to make do with them.

This note is structured as follows. In Section~\ref{sectionmodel} we give a formal description of the ARW model. In Section~\ref{section_description} we describe a dynamical procedure to move particles on the ring according to the ARW rules. This procedure is constructed in such a way that certain important properties are preserved, as shown in Section~\ref{section_propertiesproof}. Theorem~\ref{slow_phase} is proved in Section~\ref{alternate_mode}, using the cycle argument. In Section~\ref{bound_frozen} we prove a bound on the exponential moment of the number of ``bad'' blocks. In Section~\ref{section_oneblock} we prove a one-block estimate and in Section~\ref{section_markov} we use bounds on the tail of excursions of random walks to estimate the drift of the sites where there will be no particles.

\section{Description of the model} \label{sectionmodel}

In this section we give a more formal description of the model. The configuration space will be $\{0, \mathfrak{s}, 1, 2, \dots\}^{\mathbb{Z}_N}$, where $\mathfrak{s}$ is a symbol used to denote one sleeping particle. We define $\mathfrak{s}+1=2$.

With each site $x$ on $\mathbb{Z}_N$ is associated a \emph{stack of instructions} $\xi_x:=(\xi_k^x)_{k \in \mathbb{N}}$ sampled independently over $x$ and $k$ according to the following distribution
\begin{equation*}
	\xi_k^x=\begin{cases}
		\mathfrak{t}_{x,x+1} & \text{with probability } \frac{1}{2(1+\lambda)}; \\
		\mathfrak{t}_{x,x-1} & \text{with probability } \frac{1}{2(1+\lambda)}; \\
		\mathfrak{t}_{x,\mathfrak{s}} & \text{with probability } \frac{\lambda}{1+\lambda},
	\end{cases}
\end{equation*}
where each of the above transition acts on a configuration $\eta \in \{0, \mathfrak{s}, 1, 2, \dots\}^{\mathbb{Z}_N}$ in the following way
\begin{equation*}
	[\mathfrak{t}_{x,y}\eta](z)=
	\begin{cases}
		\eta(z)+1, & \text{if } z=y; \\
		\eta(z) - 1, & \text{if } z=x \\
		\eta(z), & \text{otherwise,}
	\end{cases}
	\qquad
	[\mathfrak{t}_{x,\mathfrak{s} }\eta](z)=
	\begin{cases}
		\mathfrak{s}, & \text{if } z=x \text{ and } \eta(x)=1;	\\
		\eta(z), & \text{otherwise.} 
	\end{cases}
\end{equation*}

A site $x$ is called \emph{stable} in a configuration $\eta$ if $\eta(x)=0$ or $\mathfrak{s}$, and \emph{unstable} otherwise. A configuration $\eta$ is \emph{stable} if every~$x$ is stable in $\eta$. The operation of \emph{toppling} a site~$x$ consists in updating the configuration by $\eta \mapsto \xi_{(x,k)}\eta$, with $k:= \min \{j \! : \! \text{instruction } \xi_{(x,j)} \text{ has not been used} \}$. The toppling of a site $x$ is \emph{legal} if $x$ is unstable. 

A \emph{legal sequence of topplings} is a $k$-tuple $\alpha=(x_1,x_2, \dots, x_k)$ of sites where $x_1$ is unstable, $x_2$ is unstable after $x_1$ is toppled, $x_3$ is unstable after $x_1$ and $x_2$ are toppled, etc. The \emph{odometer of $x$ in $\alpha$}, denoted by $m_{\alpha}(x)$ is the number of times the site $x$ appears in $\alpha$. A legal sequence $\alpha$ \emph{stabilizes} $\eta$ if after we topple all the sites of $\alpha$ we obtain a stable configuration.

The \emph{Abelian property} of the ARW (Lemma~2.4 of~\cite{Rolla20}) states that if $\alpha$ and $\beta$ are two legal sequences of topplings that stabilize a configuration $\eta$, then $m_{\alpha}=m_{\beta}$. Furthermore, the order in which legal topplings are performed does not affect the final configuration obtained.

For a configuration $\eta$, we define the \emph{odometer} of $\eta$ to be $m_{\eta}=m_{\alpha}$, with $\alpha$ a legal sequence of topplings that stabilize $\eta$. By the Abelian property, this definition does not depend on our choice of $\alpha$.

\section{The toppling procedure} \label{section_description} \label{initialization} \label{section_properties}

In this section, we describe a sequence of topplings in the ring $\mathbb{Z}_N$. In Section~\ref{alternate_mode}, we use this procedure as the main step of an iterative procedure to stabilize the configuration on the closed ring. 

\subsection{Preliminary steps} \label{subsection:preliminaries}

Without loss of generality, we can assume that the initial configuration has at most one particle per site. Otherwise, because $\zeta<1$, we can a.s.\ topple sites with more than one particle until there are no sites with more than one particle.

Let $a \in 2 \mathbb{N}$ be fixed (its value will be determined later) and let $K:=a^2$. For simplicity, we assume that $N=(n+2)K$ for some even integer $n$. The ring $\mathbb{Z}_N$ will be identified with the interval $\{-K/2,-K/2+1, \dots, N-K/2-1\}$ in the obvious way. For $0 \le i \le n+1$, we call the set $[iK-K/2, iK+K/2)$ the \emph{block~$i$}. 

During the procedure, we will label each particle as a \emph{free} or a \emph{carpet} particle. Free particles will be further subdivided into \emph{frozen} or \emph{thawed}. At the beginning, we declare all particles at sites $iK$, $i=0,1,\dots, n+1,$ to be free and thawed, and all the other ones to be carpet particles. We choose one free thawed particle to be the \emph{hot} particle. The criterion for picking the hot particle will be described later. 

There are some special sites, called \emph{holes}, that will be ``moved'' during the procedure. At the beginning, we define the holes to be at sites $iK$ for $i=0,1,\dots, n+1$.

A \emph{defect} or \emph{defective site} is a site that is not a hole and has no particle. When an active particle reaches a defect, we say that the defect is \emph{fixed}. In this case, the site stops being defective.

The density $\zeta$ will be chosen later in such a way that, on average, each block contains less than one vacant site (a hole or a defect).

With the assumptions and definitions above, the initial configuration $\eta$ will satisfy the following properties, which will be preserved by the toppling procedure that will be described below.

\begin{enumerate}[label=(P\arabic*),ref=(P\arabic*)]
	\item \label{P1} Each block $i$ has exactly one hole, which is located at some site in $[iK, iK + a]$.
	
	\item \label{P2} Every site contains a carpet particle, except for the holes and defective sites.
	
	\item \label{P4}If there are defective sites in $[iK-K/2,iK+K/2)$, the hole is at $iK$.
	
	\item \label{P5} Carpet particles between the hole and $iK + a$ are active.

	\item \label{P6} Free particles are always active.
	
	\item \label{P7} All free particles are at sites $iK$ or $iK + a$, except, possibly, the hot particle.
	
	\item \label{P8} There is at most one frozen free particle per block.
	
	\item \label{P9} There is a frozen free particle in block $i$ if and only if the hole is at position $iK+a$. In this case, the frozen free particle in this block is at position $iK+a$.
	
	\item \label{P10} The hot particle is free and thawed.
	
\end{enumerate}

\subsection{Attempted emission}

We will describe in this section a procedure for an \emph{attempted emission}. An attempted emission can end either in a \emph{successful emission} or in a \emph{failure}.

Suppose we choose the hot particle at block $i \in \{1,\dots,n\}$. Note that the hot particle is never chosen in blocks $0$ and $n+1$. A \emph{successful emission} (or just \emph{emission}) occurs either  when the hot particle reaches $(i+1)K$ (resp. $(i-1)K+a$) and there are no defects in block~$i+1$ (resp. $i-1$); or when there are defects in block~$i+1$ (resp. $i-1$) and the hot particle reaches a vacant site in block~$i+1$ (resp. $i-1$). We remember that, by~\ref{P2}, vacant sites are either holes or defects. In either case, block $i$ is called the \emph{emitting} block and block $i+1$ (resp $i-1$) the \emph{receiving} block. When an emission occurs, we say that the hot particle is \emph{emitted}.

Note in particular that the hot particle may be emitted when reaching the hole of block $i-1$ (resp.\ $i+1$) if this hole is vacant and if there are defects in $[(i-1)K-K/2,(i-1)K]$ (resp.\ $[(i+1)K, (i+1)K+K/2]$).

Let us describe the procedure for an attempted emission. First, we will choose the hot particle. The hot particle is chosen using the following criterion. At the beginning of an attempted emission, pick the smallest $i \in \{1,2, \dots n\}$ such that there are no defects and there is at least one thawed free particle in block~$i$. If there is no such $i$, we declare the procedure finished. After~$i$ is chosen, we pick one of the free thawed particles of block~$i$ to be the hot particle. Note that these free thawed particles must be in $iK$ or $iK+a$ by~\ref{P7}. If there is a free thawed particle at~$iK$, we choose the hot particle at $iK$; else, we choose the hot particle at $iK+a$. 

By~\ref{P6} and~\ref{P10} we can always topple the site that contains the hot particle. We will only topple the site where the hot particle is.

We have the following cases.

\textbf{There is no frozen particle at site $iK+a$:} Topple the hot particle until it is emitted or reaches the hole in $[iK,iK+a]$. If it is emitted, we finish the attempted emission. If the emission occurs because the hot particle fixed a defect, we turn the hot particle into a carpet one. Otherwise, it remains free and thawed. Note that an emission may occur if the hot particle reaches a vacant hole and there are still defects in the neighbor block. If the hot particle reaches the hole and sleeps, we turn it into a carpet particle, move the hole to the site immediately to the right, turn the carpet particle in the new location of the hole into a free thawed particle, and choose this particle to be hot (\ref{P5} ensures that the new hot particle is active). If the hole reaches $iK+a$, we stop the attempted emission and declare the free particle at $iK+a$ as frozen. If this happens, we say that the attempted emission ends in a \emph{failure} and, if possible, a new hot particle is chosen in the way described above.

If the particle leaves the hole, there are two sub-cases:

\emph{-The particle is emitted.} In this case, we finish the attempted emission.

\emph{-The particle makes an excursion and returns to the hole.} In this case, we move the hole to the leftmost site of $[iK,iK+a]$ visited by the hot particle in this excursion, turn the hot particle into a carpet one, and turn the particle at the new position of the hole into a free thawed one and turn this particle in the hot particle. (Note that, if the hot particle made an excursion to the right of the hole, the leftmost site visited is the hole itself.)

\textbf{There is a frozen free particle at $iK+a$:} In this case, by the combination of~\ref{P2},~\ref{P4}, and~\ref{P9}, there are no defects on $[iK-K/2, iK+K/2)$, and all sites in $[iK,iK+a)$ have a carpet particle. We topple the hot particle until it is emitted. If the particle visits every site in $[iK,iK+a]$ before being emitted, we move the hole to $iK$, turn the frozen particle at $iK+a$ into a carpet particle and turn the carpet particle at $iK$ into a free thawed particle. This finishes the attempted emission.

See Figure~\ref{fig_procedure} for an example of a piece of the ring after some attempted emissions have been made.

\begin{figure}
	\includegraphics[width=\textwidth]{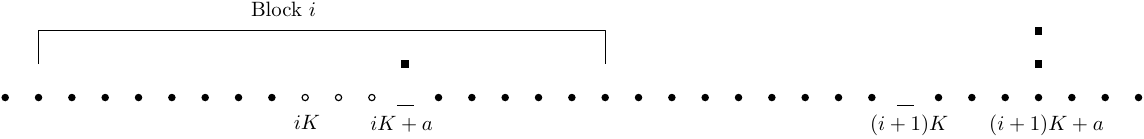}
	\caption{Example of possible state of a piece of the ring when an attempted emission is about to begin. Filled circles are active carpet particles, empty circles are sleeping carpet particles. The squares represent free particles, one of which is frozen. The holes are indicated by underlining the site. There is a frozen particle at $iK+a$.}
	\label{fig_procedure}
\end{figure}

\section{Properties preserved} \label{section_propertiesproof}

In this section, we prove that the properties~\ref{P1} to~\ref{P10} of a configuration are preserved by the procedure described in Section \ref{section_description}.
\begin{proposition}
	At the end of each attempted emission, each one of the properties~\ref{P1} to~\ref{P10} is preserved.
\end{proposition}
\begin{proof}
	This follows for the following reasons:
	
	\ref{P1} We never move the hole outside $[iK,iK+a]$, and every time we move the hole to a new location, the old location ceases to be a hole. Thus, there is only one hole at each block.
	
	\ref{P2} This happens because sites that are either defects or holes only cease to be defects and holes if a particle reaches them and is turned into a carpet particle; and with exception of the hole, we only topple sites with at least two particles.

	\ref{P4} We only move the hole after all defects in $[iK-K/2,iK+K/2)$ are fixed, because there are no defects in the block from which the hot particle is chosen.
	
	\ref{P5} If a particle fixes a defect, we relabel it as a carpet one and it remains active. If a particle sleeps, we move the hole to the right. We only move the hole to the left if all the particles between its new location and $iK+a$ are active.
	
	\ref{P6} Every time a free particle sleeps, we relabel it as a carpet particle.
	
	\ref{P7} We only move the hot particle, and every time the hot particle fixes a defect or sleeps, we relabel the hot particle as a carpet one. We only stop to move the hot particle when it is emitted or frozen. Thus, every time we choose a new hot particle, the old one is either relabeled as a carpet particle (if it fixed a defect), is frozen (at $iK+a$) or reached block $i-1$ or $i+1$.
	
	\ref{P8} If there is a frozen free particle at the block where the hot particle is chosen, the hot particle is always emitted or fixes a defect. In either case, we do not freeze any other particle.
	
	\ref{P9} This is because a particle is frozen if and only if the hole reaches $iK+a$, and is turned from frozen to carpet when the hole leaves $iK+a$.
	
	\ref{P10} We choose a free and thawed particle to be hot. When a hot particle is relabeled as a carpet one or declared frozen, a new hot particle is chosen.

\end{proof}

\section{Alternating modes} \label{alternate_mode}

We will now prove Theorem~\ref{slow_phase}. To do so, we apply cyclically the procedure of Section~\ref{section_description} in two \emph{modes}, A and B. At each mode, blocks $0$ and $n+1$ are called \emph{sinks} and blocks $n/2$ and $n/2+1$ are called \emph{sources}. At the end of each mode we change the enumeration of the blocks in a way such that the sinks of mode A are the sources of mode B and vice versa. Note that the hot particle is never chosen in the sink regions.

Let $D$ the number of defects at the end of a mode. We will define the following Condition and use it in the proof of Theorem~\ref{slow_phase}.

\begin{condition} \label{condition1}
	There are at least $\frac{7}{8}n+D$ free particles of which at most $\frac{5}{8}n$ are frozen.
\end{condition}

Condition~\ref{condition1} ensures that there is enough free thawed particles to keep the procedure running.

\begin{proposition}
	\label{PropCondition}
	There is $c>0$ depending on $\lambda$ such that, if at the beginning of a mode Condition~\ref{condition1} is satisfied, then at the beginning of the next mode Condition~\ref{condition1} will also be satisfied with probability at least $1-e^{-cn}$, for all $n$ sufficiently large, uniformly on all possible configurations at the beginning of the mode.
\end{proposition}

We will prove Proposition~\ref{PropCondition} at the end of this section.
Now we are ready to prove Theorem~\ref{slow_phase}.

\begin{proof}
	[{Proof of Theorem~\ref{slow_phase}}] We take $1-\frac{1}{4K}<\zeta<1$. Since the initial configuration is deterministic with $\zeta N$ particles, at the initial configuration there are at most $n/4$ defects, and at least $3n/4$ blocks with one free thawed particle. Thus, there will be at least one block without defects and with one free thawed particle, and we can start the toppling procedure. 
	
	We will call a mode \emph{successful} if Condition~\ref{condition1} is met at the end of the mode. We first note that, if a mode is successful, at least one jump is made. This happens because there is never more than one free thawed particle on a block which has defects. Thus, Condition~\ref{condition1} ensures that there are at least $n/4$ free particles on blocks without defects. At the end of the mode, such free thawed particles will be necessarily at the sink. Hence, at least one jump was made.

	Using Proposition~\ref{PropCondition},
	\begin{align*}
		\pr(\mathcal{J} \ge e^{\delta N}) &\ge\pr(\text{The first } e^{\delta N} \text{ modes are successful}) \\ & \ge (1-e^{-cN/K})^{e^{\delta N}}\ge 1-e^{(\delta-c/K)N}.
	\end{align*}
	The theorem follows if we take $\delta<c/K$ and $\delta'=c/K-\delta$.
\end{proof}

Let $E_n$ be the set of blocks that realize at least one successful emission during a mode, $F(E_n)$ be the number of frozen particles in $E_n$ and the end of a mode and $F(\mathbb{Z}_N)$ the total number of frozen particles in the ring $\mathbb{Z}_N$ at the end of a mode.

To prove the Proposition~\ref{PropCondition} we will make use of the following statements.
\begin{proposition} \label{frozenE_n} There exists a constant $c$ such that, if Condition~\ref{condition1} is satisfied then 
	$\pr(F(E_n) \ge \frac{n}{8}) \le e^{-cn}$ for all $n$ sufficiently large, uniformly on all possible configurations at the beginning of the mode. 
\end{proposition}

\begin{proposition} \label{PropCondition'}
	There exists a constant $c>0$ such that, if at the beginning of the mode all the holes are at positions $iK$, for $i = 1 ,\dots, n$, then $\pr(F(\mathbb{Z}_N) \ge \frac{n}{8}) \le e^{-cn}$ for all $n$ sufficiently large, uniformly on all possible configurations at the beginning of the mode.
\end{proposition}

\begin{proof}
	[Proof of Proposition~\ref{PropCondition}] 
	Note that the number of free particles minus the number of defects is conserved by the toppling procedure, therefore at the end of the mode there are still at least $\frac{7}{8}n+D$ free particles.
	
	In the first mode all the holes are at positions $iK$, for $i = 1 ,\dots, n$, hence by Proposition~\ref{PropCondition'}, the probability of having at most $n/8$ frozen particles in total on the ring at the end of the mode is at least $1-e^{-cn}$. In this event, Condition~\ref{condition1} is met at the end of the mode.
	
	Now we will consider the case from the second mode onwards. Let us show that, on the event $\{F(E_n) \le n/8\}$, Condition~\ref{condition1} is met at the end of the mode.
	
	We first argue that in this event, there will be free particles inside the sink blocks. From the second mode onward, all free thawed particles begin the mode in the source blocks, except free thawed particles in blocks which still have defects, if any. Therefore, $E_n$ is a connected set surrounding the sources, and new frozen particles outside $E_n$ must be in the two blocks adjacent to $E_n$.
	
	If $\{F(E_n) \le n/8\}$ happens, at the end of the mode, there are at most 2 new frozen particles outside $E_n$ and $5n/8$ particles that were frozen at the beginning of the mode. Since $5n/8+n/8+2<7n/8$ (if $n>16$), we conclude that the number of free thawed particles at the end of the mode is strictly higher than the number of defects. On a block with a defect there can be at most one free thawed particle; thus, the number of free thawed particles on blocks with defects is at most the number of defects, and there must have been free particles that are not frozen by the end of the mode, and this can only be achieved if free particles are emitted to a sink block.
	
	Therefore, we conclude that, in fact, $E_n$ must contain the source block and be adjacent to a sink block, and since $E_n$ is a connected collection of blocks, the number of blocks in the complement of the union of $E_n$ and the sink blocks is at most $n/2$. In particular, by~\ref{P8}, there are at most $n/2$ frozen particles outside $E_n$, hence at most $5n/8$ frozen particles in total.
	
	By Proposition~\ref{frozenE_n}, $\pr(F(E_n) < n/8)\ge 1-e^{-cn}$, and Proposition~\ref{PropCondition} is proven. 
\end{proof}

\section{Upper bound for the number of frozen particles} \label{bound_frozen}

The main goal of this Section is to prove Proposition~\ref{frozenE_n} and~\ref{PropCondition'}. In this section, the analysis is restricted to the dynamics of a single mode. 

\subsection{$\sigma$-algebras and flow of particles} \label{ssection_obs}

With each site $y \in [iK-K/2,iK+K/2)$, $i=0,i,\dots, n+1$ we associate \emph{three} independent copies of the stack of instructions, $\xi^{y}$, $\xi^{y,L}$ and $\xi^{y,R}$. Each time one of those sites is toppled we use the first unused instruction in $\xi^{y,L}$ if the hot particle is chosen at $\{(i-1)K, (i-1)K+a\}$, the first unused instruction in $\xi^y$ if the hot particle is chosen at $\{iK,iK+a\}$ or the first unused instruction in  $\xi^{y,R}$ if the hot particle is chosen at $\{(i+1)K, (i+1)K+a\}$. Thus, each site $y$ in block $i$ has a stack associated to the block $i-1$ (except for $i=0$), one stack associated to block $i$ and one stack associated to block $i+1$ (except for $i=n+1$).

We define the following $\sigma$-algebras
\begin{align*}
	\mathcal{F}_i=\sigma&(\{\xi^y, \xi^{y,R}, \xi^{y,L}: y\in (-K/2,iK-K/2)\} \\&\cup \{\xi^{y,L}, \xi^{y}: y \in [iK-K/2, iK+K/2)\} \\ &\cup\{\xi^{y,L}: y \in [iK+K/2, (i+1)K+K/2)\}).	
\end{align*} 

In words, $\mathcal{F}_i$ comprises information about stacks associated with blocks $0, 1, \dots, i$.

Let $\eta$ be the configuration of the ring at the beginning of a mode, and, for $0 \le k \le n$, define $\eta_k$ to be the restriction of $\eta$ to all the blocks $0,1,\dots, k$. Let $d'$ be the number of defects in $\eta_k$. We define, for $(m,d) \in \{0\} \times \{0,1,\dots, d'\} \cup \mathbb{N} \times \{d'\}$, $\eta'_k(m,d)$ in the following way: first, we use the toppling procedure of Section~\ref{section_description} restricted to blocks $0,1, \dots k$. After this is done, we define $\eta'_k(m,d)$ as the resulting configuration after block~$k$ receives $m+d$ particles from block $k+1$ ($d$ particles for defective sites, $m-1$ particles arrive necessarily at $kK+a$ and $1$ particle may arrive either at $kK$ if $kK$ is vacant and there are defects in $(kK-K/2,kK)$ or at $kK+a$ if not.)

Let $1 \le i \le k \le n$, and $m$ and $d$ integers for which $\eta'_k(m,d)$ is well-defined. We run the procedure described in Section~\ref{section_description} on $\eta'_k(m,d)$ and define the following random functions:
\begin{itemize}
	\item $L_i^k(m,d)$ is the number of particles that are emitted from block $i$ to block $i-1$;
	
	\item $R_i^k(m,d)$ is the number of particles emitted from block $i$ to block $i+1$;
	
	\item $D^k_i(m,d)$ is the number of particles that reached vacant sites in block~$i$ from block~$i+1$ and will not be available for attempted emissions;
	
	\item $S^k_i(m,d)$ is the number of particles frozen in block~$i$ after the procedure is finished.
\end{itemize}

By~\ref{P8}, $S_i^k(m,d)$ can only be $0$ or $1$.

We also define $M^k_i$ as the number of particles emitted from block~$i$ to block~$i-1$ that will be available for attempted emissions in block~$i-1$. 

For the sake of simplicity, we may also write $L_i$, $R_i$, $D_i$, $M_i$ and $S_i$ instead of $L_i^i$, $R_i^i$, $D_i^i$, $M_i^i$ and $S_i^i$.

A particle that is emitted from block $i$ to block $i-1$ can only fix a defect in block $i-1$, reach the vacant hole or wait at $iK+a$ to be used in an attempted emission. Therefore, we can state the following \text{mass-balance equation}
\begin{equation} \label{mass-balance}
	L_i^k(M_i^k, D_i^k)=M_{i-1}^k+D_{i-1}^k.
\end{equation}

The following proposition gives a relation between those random functions. 
\begin{proposition} \label{invariance}
	For every $1 \le i \le n$,
	\begin{align*}
		S_i(M_i^n,D_i^n)=S_i^{n}(0,0), L_i(M_i^n,D_i^n)=L_i^{n}(0,0), \\ R_i(M_i^n,D_i^n)=R_i^{n}(0,0), D_i^i(M_i^n,D_i^n)=D_i^{n}(0,0).
	\end{align*}
\end{proposition}

\begin{proof}
	For $i=n$, the result is immediate. Fix some $i \in\{1,\dots, n-1\}$. Whenever possible, the procedure described on Section~\ref{section_description} chooses the hot particle in blocks $1,2,\dots, i-1$, and when this choice is possible, it depends only on the configuration in blocks $1,2,\dots, i-1$. When the choice in those blocks is not possible, the procedure either stops or waits for a particle to come to block~$i$ from block~$i+1$. This particle will either fix a defect in block $i$, reach the vacant hole at $iK$ or wait at $iK+a$ until it is chosen to be the hot particle. By the way the procedure was designed, these particles will have the least priority to be chosen as the hot particle, and will not affect the dynamics in the first~$i$ blocks until they are chosen to be hot.
	
	The only other way that attempted emissions in block~$i+1$ can affect the outcome on blocks~$1$ to~$i$ is by using instructions on stacks of instructions in sites between $iK+a$ and $(i+1)K$, but this is fixed by using two i.i.d.\ stacks of instructions in those sites $\xi^{y,L}$ and $\xi^{y,R}$, one for particles coming from the right and other for particles coming from the left.
\end{proof}

\subsection{Bounding the exponential expectation} \label{ssection_bound}

To prove Proposition~\ref{frozenE_n}, we will use the following one-block estimate, to be proved in Section~\ref{section_oneblock}.
\begin{proposition} \label{oneblockestimate}
	~
	\begin{enumerate}[label*=(\roman*)]
		\item For any initial configuration $\eta_0$ satisfying Properties~\ref{P1} to~\ref{P10} and $K$ large enough,
		\begin{align*} 
			\sup_{\ell \ge 0,d\ge 0} \expec \Bigg[\sum_{m=0}^{\infty} e^{48 S_i(m,d)}\ind_{\{L_i(m,d)=\ell\}}\ind_{\{L_i(m,d)+R_i(m,d) \ge 1\}} ~\Bigl\vert ~ \mathcal{F}_{i-1}\Bigg] \leq e^4.
		\end{align*}
		where the sum runs over all $m$ for which $\eta'_i(m,d)$ is defined.
		
		\item For any initial configuration $\eta_0$ satisfying Properties~\ref{P1} to~\ref{P10} and with holes only at positions $iK$, $1 \le i \le n$ and $K$ large enough,
		\begin{align*} 
			\sup_{\ell \ge 0,d\ge 0} \expec\Bigg[\sum_{m=0}^{\infty} e^{48 S_i(m,d)}\ind_{\{L_i(m,d)=\ell\}} ~\Bigl\vert ~ \mathcal{F}_{i-1}\Bigg] \leq e^4,
		\end{align*}
		where the sum runs over all $m$ for which $\eta'_i(m,d)$ is defined.
	\end{enumerate}
\end{proposition}

\begin{proof}
	[Proof of Proposition~\ref{frozenE_n}]
	
	Denote, for simplicity, $G^i_{m,d}$ as the event $\{(R_i+L_i)(m,d)\ge 1\}$. We will first show that $ \expec[e^{48 F (E_n)}]\le(n+1)4^ne^{4n}.$
	
	By definition, $F(E_n)$ is the number of particles that are frozen on the collection of blocks that perform at least one emission. Also, at the beginning of the procedure there are at most $n$ free particles (one at each block), and the number of free particles never increases. Since $\zeta > 1-1/(4K)$, we can bound the number of defects by $n$. Therefore, using the definition of $S_i$, Proposition~\ref{invariance} and the mass-balance equation~\eqref{mass-balance}:
	\begin{flalign*}
		\expec[e^{48 F(E_n)}] &= \expec \Bigg[\sum_{\substack{d_0, \dots d_{n} \\ m_0, \dots, m_{n}}} \prod_{i=1}^{n} e^{48 S_i(m_i,d_i)}\ind_{\{M_i^n=m_i\}}\ind_{\{D_i^n=d_i\}} \ind_{\{L_i(m_i,d_i)=m_{i-1}+d_{i-1}\}}\ind_{G^i_{m_i,d_i} } \Bigg] 
		\\& \le \sum_{\substack{d_0, \dots d_{n}\\ d_0+\dots+ d_{n}\le n}}\expec \Bigg[\sum_{m_0=0}^{n} \sum_{m_1,\dots, m_{n}} \prod_{i=1}^n e^{48 S_i(m_i,d_i)} \ind_{\{L_i(m_i,d_i)=m_{i-1}+d_{i-1}\}}\ind_{G^i_{m_i,d_i} }\Bigg].
	\end{flalign*}
	
	The number of non-negative integer solutions of the inequality $d_0+d_1+\dots+d_{n}\le n$ is bounded above by $4^n$. Thus,
	\begin{align*} 
		\expec [e^{48 F(E_n)}] \le 4^n \cdot \sup_{\{d_0, \dots, d_{n}\}} \expec \Bigg[\sum_{m_0=0}^{n} \sum_{m_1,\dots, m_{n}} \prod_{i=1}^{n} e^{48 S_i(m_i,d_i)} \ind_{\{L_i(m_i,d_i)=m_{i-1}+d_{i-1}\}}\ind_{G^i_{m_i,d_i} }\Bigg].
	\end{align*}
	
	We will show by induction that, for every positive integer $r\le n$, and every possible choice of $d_0, \dots d_n$,
	\begin{align*}
		\expec \Bigg[\sum_{m_0=0}^{n} \sum_{m_1, \dots, m_r} \prod_{i=1}^r e^{48 S_i(m_i,d_i)} \ \ind_{\{L_i(m_i,d_i)=m_{i-1}+d_{i-1}\}}\ind_{G^i_{m_i,d_i} }\Bigg] \le (n+1)e^{4 r}.
	\end{align*}
	
	Indeed, for $r=1$ the expectation above is bounded by $n+1$, and the base case of induction is finished. Suppose the inequality is valid for every integer between $1$ and $r-1$. By conditioning on $\mathcal{F}_{r-1}$ and using item (i) of Proposition~\ref{oneblockestimate},
	\begin{align*}
		& \!\!\!\!\!\!\!\!\!\!\!\!
		\expec \Bigg[\sum_{m_0=0}^{n} \cdots \sum_{m_r} \prod_{i=1}^r e^{48 S_i(m_i,d_i)} \ind_{\{L_i(m_i,d_i)=m_{i-1}+d_{i-1}\}}\ind_{G^i_{m_i,d_i} }\Bigg] \\
		&\le \expec \Bigg[\sum_{m_0} \sum_{m_1} \cdots \sum_{m_{r-1}} \prod_{i=1}^{r-1} e^{48 S_i(m_i,d_i)}\ind_{\{L_i(m_i,d_i)=m_{i-1}+d_{i-1}\}} \ind_{G^i_{m_i,d_i} } \Bigg] \\& \qquad \cdot \sup_{\ell \ge 0, d \ge 0} \expec\Bigg[\sum_{m_r} e^{48 S_r(m_r,d)}\ind_{\{L_r(m_r,d)=\ell\}}\ind_{G^r_{m_r,d_r}} \mid \mathcal{F}_{r-1}\Bigg] \\ &\le e^4 \expec \Bigg[ \sum_{m_0}^n\cdots \sum_{m_{r-1}} \prod_{i=1}^{r-1} e^{48 S_i(m_i,d_i)}\ind_{\{L_i(m_i,d_i)=m_{i-1}\}} \ind_{G^i_{m_i,d_i} } \Bigg] \\& \le (n+1)e^{4}e^{4(r-1)}=(n+1)e^{4 r}.
	\end{align*}
	Therefore, 
	$\expec [e^{48 F(E_n)}]\le (n+1)4^n e^{4 n}.$
	
	To conclude the bound, we note that, by Markov's inequality,
	\begin{align*}\pr(F(E_n) \ge n/8)&\le e^{-6 n}\expec e^{48\text F(E_n)} \\ &\le e^{-6n} \cdot (n+1) 4^n e^{4 n} \le e^{-cn}
	\end{align*}
	for a constant $c>0$ chosen properly and $n$ sufficiently large.
\end{proof}

\begin{proof}[Proof of Proposition~\ref{PropCondition'}] The proof of this proposition is completely analogous to the proof of Proposition~\ref{frozenE_n}, but summing over all the blocks, not only the ones that realize at least one emission, and using item~$(ii)$ of Proposition~\ref{oneblockestimate}, because in the first mode all holes are at positions $iK$, $i=1,2,. \dots, n$.
\end{proof}

\section{One-block estimate} \label{section_oneblock}

This section is devoted to proving Proposition~\ref{oneblockestimate}. We will begin by proving item $(i)$.

Let $i$ be fixed. We will consider two cases. First, we will analyze the case where there are defects in $[iK-K/2,iK+K/2)$ after we stabilize $\eta'_i(m,d)$ using the procedure described in Section~\ref{section_description}. The only value of $m$ compatible with the definition of $\eta'_k$ in this case is $m=0$, and, by Properties~\ref{P4} and~\ref{P9}, $S_i(0,d)=0$. It follows that
\begin{align*} 
	\sup_{\ell \ge 0,d\ge 0}& \expec\Bigg[\sum_{m=0}^{\infty} e^{48 S_i(m,d)}\ind_{\{L_i(m,d)=\ell\}}\ind_{G^i_{m,d}} \mid \mathcal{F}_{i-1}\Bigg] \\ &= 
	\sup_{\ell \ge 0,d\ge 0} \expec[ e^{48 S_i(0,d)}\ind_{\{L_i(0,d)=\ell\}}\ind_{G^i_{0,d}} \mid \mathcal{F}_{i-1}] \le 1.
\end{align*} 

We now analyze the remaining case, i.e., when there are no defects in $[iK-K/2,iK+K/2)$ after we stabilize $\eta'_i(m,d)$ using the procedure described in Section~\ref{section_description}. Thus, we need only to analyze the case where $d=0$.

We will use from now on a slight abuse of notation and denote by $L_i(j), S_i(j)$ and $\ind_{G^i_{j}}$ the values of the quantities $L_i,S_i$ and $\ind_{G^i}$ after $j$ attempted emissions in block $i$.

For each particle added at $iK+a$, at least one attempted emission will be made. The only other way an attempted emission may be triggered is when a particle reaches the vacant hole at $iK$. In other words, for each $m$ there is at least one corresponding $j$. Therefore,
\begin{align*}\sum_{m=0}^{\infty} e^{48 S_i(m,d)}\ind_{\{L_i(m,d)=\ell\}}\ind_{G^i_{m,d}} \le\sum_{j=0}^{\infty} e^{48 S_i(j)}\ind_{\{L_i(j)=\ell\}}\ind_{G^i_{j}},
\end{align*}
and it is sufficient to show that
\begin{align} \label{expecwithj}
	\sup_{\ell \ge 0}\expec \Bigg[\sum_{j=0}^{\infty} e^{48 S_i(j)}\ind_{\{L_i(j)=\ell\}}\ind_{G^i_{j}} \mid \mathcal{F}_{i-1}\Bigg] \le e^4.
\end{align}

For simplicity, as $i$ is fixed, we will write $S$, $R$, $L$ and $G$ instead of $S_i$, $R_i$, $L_i$ and $G^i$. We will denote by $\mathcal{G}_j$ the $\sigma$-field generated by the information revealed up to the $j$-th attempted emission plus the information contained in $\mathcal{F}_i$.

\begin{definition}
	The process $(H(j))_{j\in \mathbb{N}}$, taking values in $\{0, 1, \dots, a\}$, is defined as $H(j)=v$ if the hole is at position $iK+v$ after the $j$-th attempted emission. 
\end{definition}

It is immediate that $H(j)$ is $\mathcal{G}_j$-measurable. For brevity we will define $\tilde{\pr}[\cdot]:= \pr[\cdot \mid \mathcal{F}_{i-1}]$ and $\tilde{\expec}[\cdot]:= \expec[\cdot \mid \mathcal{F}_{i-1}]$. To prove Proposition~\ref{oneblockestimate} we will make use of the following lemmas.

\begin{lemma} \label{lemma_positionhole}
	If $K$ is big enough, then, for every $j\ge 1$, the following inequalities hold:
	\begin{equation} \label{lemma_badhole}
		\tilde{\pr}[H(j)> a/2 \mid \mathcal{G}_{j-1}]\ind_{\{H(j-1)\in [0,a/2]\cup \{a\}\}} \le e^{-150},
	\end{equation}
	\begin{equation} \label{lemma_holemiddle}
		\tilde{\pr}(a/2<H(j)<a\mid \mathcal{G}_{j-1})<e^{-150}.
	\end{equation}
\end{lemma}

\begin{remark} \label{remark_lemma}
	The previous lemma is trivially true for $j=0$ if the initial configuration has $H(0)=0$.
\end{remark}

Lemma~\ref{lemma_positionhole} says that, if $K$ is large enough, we can make the hole process to have a drift to the left.

\begin{lemma} \label{lemma_emission}
	For every $j\ge 0$, if $K$ is big enough then the following bound holds almost surely:
	\[\tilde{\pr}(L(j+2)>L(j) \mid\mathcal{G}_{j})\ge \frac{1}{5}.\]
\end{lemma}

For every integer $\ell \ge 0$, denote by $\tau_{\ell}$ the first $j$ for which $L(j)=\ell$, with $\tau_{-1}:=-1$. We have that $\tau_{\ell+1}> \tau_{\ell}$, because every attempted emission will result in at most one particle emitted to the left. We can rewrite~\eqref{expecwithj} to obtain
\begin{align} \expec &\left[ \sum_{j=\tau_{\ell-1}+1}^{\infty} e^{48 S(j)} \ind_{\{L(j)=\ell\}} \ind_{G_j} \mid \mathcal{F}_{i-1}\right] \nonumber \\&
	\le \sum_{j=\tau_{\ell-1}+1}^{\infty} \tilde{\expec}[\ind_{\{L(j)=\ell\}}\ind_{\{S(j)=0\}}+e^{48}\ind_{\{L(j)=\ell\}}\ind_{G_j}\ind_{\{S(j)=1\}}] \nonumber\\&
	\le \tilde{\expec}[\tau_{\ell+1}-\tau_{\ell}]+\sum_{k=1}^{\infty} \tilde{\expec}[e^{48}\ind_{\{L(\tau_{\ell-1}+k)=\ell\}}\ind_{G_{\tau_{\ell-1}+k}}\ind_{\{S(\tau_{\ell-1}+k)=1\}}]. \label{sum_complete}
\end{align}

Using Lemma~{\ref{lemma_emission}}, we get
\begin{align} \label{expec_stopping}
	\tilde{\expec}[\tau_{\ell+1}-\tau_{\ell}] =\sum_{r=0}^{\infty} \tilde{\pr}(\tau_{\ell+1}-\tau_\ell>r)\le\sum_{r=0}^{\infty}\left(\frac{4}{5}\right)^{\lfloor r/2\rfloor}<10.
\end{align}

We split the sum in~\eqref{sum_complete} at $k_0:=250$. By Lemma~\ref{lemma_emission}, 
\begin{align} \label{sum_big}
	\sum_{k=k_0}^{\infty}& \tilde{\expec}[e^{48}\ind_{\{L(\tau_{\ell-1}+k)=\ell\}}\ind_{G_{\tau_{\ell-1}+k}}\ind_{\{S(\tau_{\ell-1}+k)=1\}}] \nonumber
	\\& \le \sum_{k=k_0}^{\infty} e^{48}\tilde{\pr}(L(\tau_{\ell-1}+k)=\ell) \le \sum_{k=k_0}^{+\infty}e^{48}\left(\frac{4}{5}\right)^{\lfloor k/2\rfloor}<1.
\end{align}

We will now prove that, for every $\ell \ge 0$ and $1 \le k \le k_0$,
\begin{align} \label{bound_expec}
	\tilde{\expec}[\ind_{\{L(\tau_{\ell-1}+k)=\ell\}}\ind_{G_{\tau_{\ell-1}+k}}\ind_{\{S(\tau_{\ell-1}+k)=1\}}]<k_0^{-1}e^{-48}.
\end{align}

Let us analyze first the case $\ell \ge 1$. In this case, obviously $\ind_{G_{\tau_{\ell-1}}}=1$. First, we consider $k=1$. Note that if $S(\tau_{\ell-1}+1)=1$ and $L(\tau_{\ell-1}+1)=\ell$, then the $(\tau_{\ell-1}+1)$-th attempted emission is successful and $H(\tau_{\ell-1})=a$. Using~\eqref{lemma_badhole}, after conditioning on $\mathcal{G}_{\tau{\ell-1}}$,
\begin{align*}
	\tilde{\expec}&[\ind_{\{L(\tau_{\ell-1}+1)=\ell\}}\ind_{\{S(\tau_{\ell-1}+1)=1\}}]
	\\ &	\le \tilde{\expec}[\ind_{\{L(\tau_{\ell-1}+1)=\ell\}}\ind_{\{H(\tau_{\ell-1}+1)=a\}}\ind_{\{H(\tau_{\ell-1})=a\}}] \\& <e^{-150}<k_0^{-1}e^{-48}.
\end{align*}
Still assuming $\ell\ge 1$, for $2 \le k\le k_0$, we have that
\begin{eqnarray*}
	\tilde{\mathbb{P}}(S(\tau_{\ell-1}+k)=0) & \ge & \tilde{\mathbb{E}}[\mathbb{E}[\mathbf{1}_{\{H(\tau_{\ell-1}+k) \in[0,a/2]\}}]\mid \mathcal{G}_{\tau_{\ell-1}+k-1}] 
	\\ & \ge & \tilde{\mathbb{E}}[(1-e^{-150})\mathbf{1}_{\{H(\tau_{\ell-1}+k-1) \in[0,a/2]\}}]
	\\ &=&(1-e^{-150})\tilde{\mathbb{E}}[\mathbb{E}[\mathbf{1}_{\{H(\tau_{\ell-1}+k-1) \in[0,a/2]\}}\mid \mathcal{G}_{\tau_{\ell-1}+k-2}]] \\ &\ge& (1-e^{-150})^2 > 1-\frac{1}{k_0e^{48}},
\end{eqnarray*}

\noindent where in the second line we used (\ref{lemma_positionhole}) and in the last line we used (\ref{lemma_holemiddle}).

This concludes the case $\ell \ge 1$. Now let us analyze the case $\ell = 0$. We remember that we defined $\tau_{-1}=-1$. When $k=1$, $\ind_{G_0}=0$ trivially and~\eqref{bound_expec} is valid. For $\ell=0$ and $k=2$, if the first attempted emission was not successful, then $\ind_{G_1}=0$. The only way the first attempt is successful, $S(1)=1$ and $L(1)=0$ occurs if $S(0)=1$. By~\eqref{lemma_badhole}, the probability of $S(0)=1$ and $S(1)=1$ is less than $e^{-150}$. Either way, we conclude that~\eqref{bound_expec} holds when $k=1$.

For $\ell=0$ and $2 < k \le k_0$ we use~\eqref{lemma_badhole} and~\eqref{lemma_holemiddle} in the same way as before to conclude that $\tilde{\pr}(S(k)=0)>1-\frac{1}{k_0e^{48}}$.
Combining~\eqref{expec_stopping}, \eqref{sum_big} and~\eqref{bound_expec}, we conclude~\eqref{expecwithj}.

The proof of $(ii)$ for $\ell \ge 1$ is almost identical to the proof of $(i)$. In the case $\ell=0$, we cannot use Lemma~\ref{lemma_positionhole} directly. But, since by hypothesis the initial configuration has holes only at $iK$, we can use Remark~\ref{remark_lemma} instead of Lemma~\ref{lemma_positionhole} to show that, for $1 \le k \le k_0$, 
$\tilde{\expec}[\ind_{\{L(\tau_{\ell-1}+k)=\ell\}}\ind_{\{S(\tau_{\ell-1}+k)=1\}}]<\frac{1}{k_0e^{48}}$ and proceed as in $(i)$.

\section{Estimating the hole drift} \label{section_markov}

This section is devoted to prove Lemmas~\ref{lemma_positionhole} and~\ref{lemma_emission}.

The movement of the hole has a drift to the right, at least if it is in the interval $[iK, iK+a/3]$, if $\lambda$ is large. We then choose $a$ large enough such that the hole has a drift to the left as soon as it is at the right of $iK+a/3$. Then, as long as there is no emission, the hole will be ``trapped'' between those two drifts, and would take an exponentially large number of trials to reach $iK+a$. On the other hand, the event of no emission occurring after an exponential number of trials is unlikely, because the probability of an emission is roughly $1/(\lambda a^2)$, and this is of order $1/a^2$ if $a$ is much larger than $\lambda$.

The probability of occurrence of an emission to the left (or to the right) is bounded above by

\[\delta:=\left(\frac{1}{2}\right)\left(\frac{1}{1+\lambda}\right)\left(\frac{1}{K-a}\right).\]
For each $v \in \mathbb{N}$, $v \le a$, define
\[\tilde{Y}_v=\begin{cases}
	+1, & \text{with probability } \frac{\lambda}{1+\lambda};
	\\ 0, & \text{with probability } \frac{1/2}{1+\lambda}+\delta;
	\\ -k, & \text{with probability } \frac{1/2}{1+\lambda}\frac{1}{k(k+1)} \text{ for } k=1,2,\dots, v-1;
	\\ -v & \text{with probability } \frac{1/2}{1+\lambda}-\delta-\sum_{k=1}^{v-1} \frac{1/2}{1+\lambda}\frac{1}{k(k+1)}.
\end{cases}\]

Those random variables dominate stochastically the change of the position of the hole, when it is at $iK+v$, conditioned to the event that there are no defects inside the block.

\begin{lemma} \label{lemma_expecY}
	If $a$ is large enough and $v \ge a/3$, $\expec[\tilde{Y}_v]\le-40$.
\end{lemma}

\begin{proof}
	Note that
	\begin{align*}
		\expec[\tilde{Y}_v]
		&\le 1 - (2(\lambda+1))^{-1}\sum_{k=1}^{a/3}\frac{1}{k(k+1)}\cdot k \\
		&\le 1 - (\log(a/3)-\log 2)(2(\lambda+1))^{-1}.
		%&\le (v+1) \delta + 1 - (\lambda+1)^{-1}-(2(\lambda+1))^{-1}\sum_{k=1}^{a/3}\frac{1}{k(k+1)}\cdot k \\
		% &\le (v+1)\delta + 1- (\lambda+1)^{-1} - (\log(a/3)-\log 2)(2(\lambda+1))^{-1}.
	\end{align*}
	The lemma follows if we take $a$ greater than $6\exp(120(\lambda+1))$.
\end{proof}

Thus, the hole has a drift to the left, when $v \ge a/3$, at least when no emission is made. This is because the excursion of a simple random walk has a heavy tail, and, if we take $a$ large enough, at least one toppling of the hot particle will result in an excursion instead of the particle going to sleep.

\begin{lemma} \label{sum_Y}
	Let $(\tilde{Y}_{a/3}(i))_{i \in \mathbb{N}}$ be a sequence of i.i.d.\ random variables with common distribution $\tilde{Y}_{a/3}$. Then there is a constant $\alpha > 0$ such that
	\begin{align*}\pr \left(\sum_{i=1}^{a/6}\tilde{Y}_{a/3}(i) \ge -a/6 \right) \le e^{-\alpha a}
		\quad
		\text{ and }
		\quad
		\pr \left(\sum_{i=1}^{a/2}\tilde{Y}_{a/3}(i) \ge -2a/3 \right) \le e^{-\alpha a}.
	\end{align*}
\end{lemma}

\begin{proof}
	Let $(Y^i)_{i=1}^{+\infty}$ be a sequence of bounded i.i.d.\ random variables with expectation $\nu<0$. Fix $b>0$. By Chernoff's concentration inequality, for every $\gamma>0$ such that $\gamma b$ is an integer with $\nu < -\gamma^{-1}$, there is a constant $\alpha '>0$ such that
	\[\mathbb{P}\left(\sum_{i=0}^{\gamma b}Y^i>-b\right)\le e^{-\alpha '\gamma b}.\]	
	By Lemma~\ref{lemma_expecY}, $\mathbb{E}[\tilde{Y}_{a/3}]<-40$. We can then take $b=a/6$ and $\gamma=1$ to obtain the first inequality of the lemma, and $b=-2a/3$, $\gamma=3/4$ to obtain the second one.	
\end{proof}

We call a \emph{step} of the toppling procedure when the hot particle starts in the hole and either: $(i)$ falls asleep; $(ii)$ is emitted; or $(iii)$ leaves the hole and returns to it. Let $T_j$ be the number of steps taken by the procedure between the $(j-1)$-st and the $j$-th attempted emissions.

\begin{lemma} \label{lemma_steps}
	If $a$ is large enough, for all $v\in \{0,\dots, a\}$,
	$\pr(T_j>a^3 \medspace | \medspace H(j-1)=v)<a^{-1}$.
	Moreover, if $v < a/2$, then
	$\pr(T_j<a/2 \medspace | \medspace H(j-1)=v)<a^{-1}$.
\end{lemma}

The exact value of the constants $1/3$ and $1/2$ is not important. We need the ``gap between'' $a/3$ and $a/2$ to be able to use the drift to the left of the hole and still have a good number of trials before the hole be able to reach $iK+a$. 

\begin{proof}[Proof of Lemma~\ref{lemma_steps}]
	First we note that if the hot particle is at the hole, the probability that the next instruction will be a jump instruction is $(\lambda+1)^{-1}$. Therefore, the probability of an emission is at least $((\lambda+1)(3K/2+a))^{-1}$. The first inequality comes then from the fact that, if $a:=\sqrt{K}$ is big enough, $(1-((\lambda+1)(3K/2+a))^{-1})^{a^3}<a^{-1}$.
	
	If $v < a/2$, then the first $a/2$ steps of an attempted emission cannot end with a failure. If $T_j< a/2$, then the $j$th attempted emission ended in an successful emission. The probability of occurrence of an emission is at most $((\lambda+1)(K-a))^{-1}$, then an union bound gives that the probability of an emission occurring in the first $a/2$ steps is at most $(a/2)((\lambda+1)K)^{-1}\le a^{-1}$ and the second inequality of the lemma follows.
\end{proof}	
\begin{proof}[Proof of Lemma~\ref{lemma_positionhole}]
	We begin by proving~\eqref{lemma_badhole}. It is sufficient to prove that, for every $v \in \{0,a/2\} \cup \{a\}$, $\pr(H(j)>a/2 \medspace| \medspace H(j-1)=v)< \frac{4}{a}<e^{-150}$. We split in three cases.
	
	First, we consider $v=a$. In this case, by~\ref{P9} every site in the block has a carpet or a frozen particle, and the hot particle is performing a simple random walk. If the hot particle visits every site in $[0,a]$ before being emitted, the hole is reset to position $0$. The probability of this occurring is at least $\frac{K/2-a}{K/2}=1-\frac{2}{a}.$

	Then, we look at the case $v \le a/3$. By Lemma~\ref{lemma_steps}, it is immediate that $\tilde{\pr}(H(j)>a/2 \text{ and } T_j\ge a^3 \medspace| \medspace H(j-1)=v)< \frac{1}{a}$. If $\{H(j)>a/2\} \cap \{T_j<a^3\}$ occurs, one of the first $a^3$ times the hole is at $a/3$ it moves to $a/2$ before returning to the left of $a/3$. This probability is bounded above by the probability that the sum of $a/6$ independent copies of $\tilde{Y}_{a/3}$ is at least $0$. Using Lemma~\ref{sum_Y}, $\tilde{\pr}(H(j)>a/2 \text{ and } T_j < a^3 \medspace|\medspace H(j-1)=v)< \frac{1}{a}$. The case follows by an union bound. 
	
	Finally, we consider $a/3<v \le a/2$. A failure cannot occur in the first $a/2-1$ steps. The probability that an emission occurs and $T_j<a/2$ is bounded above by $1/a$ by Lemma~\ref{lemma_steps}. If an emission does not occur in the first $a/2-1$ steps and the hole is always at the right of $a/3$, the movement of the hole is stochastically dominated by $\tilde{Y}_{a/3}$. Thus the probability of $T_j<a/2$ and the hole be always at the right of $a/3$ is bounded above by the probability that the sum of $a/6$ copies of $\tilde{Y}_{a/3}$ is at least $-a/6$, and this probability is at most $e^{-\alpha a}$ by Lemma~\ref{sum_Y}. Once the hole is at $a/3$ or to the left of $a/3$ we proceed as in the previous case. Hence, we get $\tilde{\pr}(H(j)>a/2 \text{ and } T_j < a/2\medspace|\medspace H(j-1)=v)< \frac{1}{a}$, and $\tilde{\pr}(H(j)>a/2 \text{ and } a/2\le T_j < a^3 \medspace| \medspace H(j-1)=v)< \frac{1}{a}$. By Lemma~\ref{lemma_steps}, $\tilde{\pr}(H(j)>a/2 \text{ and } T_j > a^3 \medspace|\medspace H(j-1)=v)< \frac{1}{a}$,
	and we conclude by an union bound.

	It remains to prove~\eqref{lemma_holemiddle}. Again, it is sufficient to condition on $H(j-1)$. If $H(j-1) \in [0,a/2] \cup \{a\}$, we can use~\eqref{lemma_badhole} to bound the probability of $H(j) \in (a/2,a)$. So let us suppose $H(j-1) \in (a/2,a)$. Since when a failure occurs we have $H(j)=a$, we only need to consider the case of an emission. The probability that an emission occurs in the first $a/2$ steps is bounded above by $\frac{1}{a}$ by the same argument used in the proof of Lemma~\ref{lemma_steps}. The probability that a successful emission does not occur and the hole is never in $[0,a/3]$ in the first $a/2$ steps is at most $e^{-\alpha a}$ by Lemma~\ref{sum_Y}. Once the hole reaches $[0,a/3]$ we can use Lemma~\ref{sum_Y} again as in the proof of~\eqref{lemma_badhole} to conclude that the probability the hole reaches $(a/2,a)$ is at most $e^{-\alpha a}$. An union bound finishes the proof.
\end{proof}

\begin{proof}[Proof of Lemma \ref{lemma_emission}]
	We notice that if an attempted emission ends in a failure, the next hot particle will be emitted. Thus, for each two attempted emissions, at least one will be successful. The probability that such an emission occurs to the left is at least $\frac{((i+1)K/2)-(iK+a)}{((i+1)K/2)-(iK-3K/2+1)}=\frac{1}{4}-\frac{1}{\sqrt{K}}\ge\frac{1}{5}$.
\end{proof}

\section*{Acknowledgment} The authors are grateful to a reviewer for thoughtful and constructive remarks which led to a considerable improvement in the presentation.

\section*{Funding information} B.~N.~B.~L.~ was partially supported by CNPq grant 315861/2023-1. L.~R.~ was partially supported by FAPESP grant 2023/13453-5. C.~T.~ was partially supported by CAPES grant 88887.595894/2020-00, FAPEMIG grant 15435 and CNPq grant 151432/2024-4.

\section*{Competing Interests}
There were no competing interests to declare during the preparation or publication process
of this article.

\bibliographystyle{abbrv}
\bibliography{refbib}

\end{document}